\documentclass[preprint,12pt]{elsarticle}
\usepackage{amsthm,amssymb}
\usepackage{amsmath}
\usepackage{color}
\usepackage{bm}

{\theoremstyle{plain}%
  \newtheorem{theorem}{Theorem}[section]
  \newtheorem{corollary}[theorem]{Corollary}
  \newtheorem{proposition}[theorem]{Proposition}
}
{\theoremstyle{remark}

}
{\theoremstyle{definition}

}
\newtheorem*{thmA}{Theorem A}
\newtheorem*{thmB}{Theorem B}

\begin{document}
\renewcommand{\theequation}{\arabic{equation}}
\begin{frontmatter}

\title{Some special Euler sums and $\zeta^\star(r+2,\{2\}^n)$}

\author{Kwang-Wu Chen}
\address{Department of Mathematics, University of Taipei}
\address{No. $1$,  Ai-Guo West Road, Taipei $10048$, Taiwan}
\ead{kwchen@uTaipei.edu.tw}

\author{Minking Eie}
\address{Department of Mathematics, National Chung Cheng University}
\address{$168$ University Rd., Minhsiung, Chiayi $62102$, Taiwan}
\ead{minking@math.ccu.edu.tw}
\date{\today}

\begin{abstract}
\setlength{\baselineskip}{14pt}
In this paper, we investigate the Euler sums
$$
G_{n+2}(p,q)=\sum_{1\leq k_1<k_2<\cdots<k_{p+1}}\frac1{k_1k_2\cdots k_pk_{p+1}^{n+2}}
\sum_{1\leq\ell_1\leq\ell_2\leq\cdots\leq\ell_q\leq k_{p+1}}\frac1{\ell_1\ell_2\cdots\ell_q}.
$$
We give another two representations, a reflection formula, and some other properties.
Then we use these results to calculate $\zeta^\star(r+2,\{2\}^n)$, for $r=0,1,2$,
as our applications.
\end{abstract}

\end{frontmatter}

\setlength{\baselineskip}{18pt}
\section{Introduction}
The multiple zeta values (MZVs) and the multiple zeta-star values (MZSVs) are defined by 
\cite{Hoff0, IKOO, Mun, Zag0}
$$
\zeta(\alpha_1,\alpha_2,\ldots,\alpha_r) = \sum_{1\leq k_1<k_2<\cdots<k_r}
k_1^{-\alpha_1}k_2^{-\alpha_2}\cdots k_r^{-\alpha_r}
$$
and
$$
\zeta^\star(\alpha_1,\alpha_2,\ldots,\alpha_r) 
= \sum_{1\leq k_1\leq k_2\leq\cdots\leq k_r}
k_1^{-\alpha_1}k_2^{-\alpha_2}\cdots k_r^{-\alpha_r}
$$
with positive integers $\alpha_1,\alpha_2,\ldots,\alpha_r$ 
and $\alpha_r\geq 2$ for the sake of convergence. 
The numbers $r$ and $|\bm\alpha|=\alpha_1+\alpha_2+\cdots+\alpha_r$
are the depth and weight of $\zeta(\bm\alpha)$. For our convenience,
we let $\{a\}^k$ be $k$ repetitions of $a$, for example, 
$\zeta(\{2\}^3)=\zeta(2,2,2)$.

MZVs of length one and two were already known to Euler. 
A systematic study of MZVs began in the early
1990s with the works of Hoffman \cite{Hoff0} and Zagier \cite{Zag0}. 
Thereafter these numbers have emerged in several mathematical 
areas including algebraic geometry, Lie
group theory, advanced algebra, and combinatorics. 
See \cite{EieB, Zhao} for introductory reviews.

The finite MZVs and MZSVs are defined as follows.
\begin{eqnarray*}
\zeta_n(\alpha_1,\alpha_2,\ldots,\alpha_r)&=&
\sum_{1\leq k_1<k_2<\cdots<k_r\leq n}
k_1^{-\alpha_1}k_2^{-\alpha_2}\cdots k_r^{-\alpha_r},\\
\zeta^\star_n(\alpha_1,\alpha_2,\ldots,\alpha_r)
&=&\sum_{1\leq k_1\leq k_2\leq\cdots\leq k_r\leq n}
k_1^{-\alpha_1}k_2^{-\alpha_2}\cdots k_r^{-\alpha_r}.
\end{eqnarray*}
The generalized harmonic numbers $H_n^{(s)}$ of order $s$ which are defined by 
$H_n^{(s)}=\sum^n_{j=1}j^{-s}$. It is known that (ref. \cite[Lemma 1]{Chen})
\begin{eqnarray*}
\zeta_n(\{1\}^m) &=& P_m(H_n^{(1)},-H_n^{(2)},\ldots,(-1)^{m+1}H_n^{(m)}),\\
\zeta_n^\star(\{1\}^m) &=& P_m(H_n^{(1)},H_n^{(2)},\ldots,H_n^{(m)}),
\end{eqnarray*}
where the modified Bell polynomials $P_m(x_1,x_2,\ldots,x_m)$ are defined by \cite{Chen, CC1, CC2}
$$
\exp\left(\sum^\infty_{k=1}\frac{x_k}{k}z^k\right)
=\sum^\infty_{m=0}P_m(x_1,x_2,\ldots,x_m)z^m.
$$
Recently Choi \cite{Choi} and Hoffman \cite{Hoff} investigated some special cases of 
the following Euler sums
$$
\sum_{m=1}^\infty\frac{\zeta_m(\{1\}^p)\zeta^\star_m(\{1\}^q)}
{m^{\alpha_1}(m+1)^{\alpha_2}\cdots(m+r-1)^{\alpha_r}}.
$$
In this paper, we consider the little different Euler sums:
\begin{equation}\label{eq1.1} 
G_{n+2}(p,q)=\sum_{1\leq k_1<k_2<\cdots<k_{p+1}}\frac1{k_1k_2\cdots k_pk_{p+1}^{n+2}}
\sum_{1\leq\ell_1\leq\ell_2\leq\cdots\leq\ell_q\leq k_{p+1}}\frac1{\ell_1\ell_2\cdots\ell_q},
\end{equation}
where $n,p,q$ are nonnegative integers. 
It is noted that $G_{n+2}(p,0)=\zeta(\{1\}^p,n+2)$ and 
$G_{n+2}(0,q)=\zeta^\star(\{1\}^q,n+2)$. 

This Euler sums have the other two representations:
\begin{thmA}
\begin{eqnarray*}
G_{n+2}(p,q)&=&\sum^{p+q+1}_{r=p+1}{r-1\choose p}
\sum_{|\alpha|=p+q+1}\zeta(\alpha_1,\ldots,\alpha_r+n+1)\\
&=&\frac1{p!q!n!}\int_{E_2} 
\left(\log\frac1{1-t_1}\right)^p\left(\log\frac1{1-t_2}\right)^q
\left(\log\frac{t_2}{t_1}\right)^n\frac{dt_1dt_2}{(1-t_1)t_2}.
\end{eqnarray*}
\end{thmA}
The function $G_{n+2}(p,q)$ has a reflection formula as follows.
\begin{thmB}
For a pair of positive integers $p$, $q$, and an integer $k\geq 0$,
we have
$$
G_{k+3}(p-1,q)+(-1)^kG_{k+3}(q-1,p)
=\sum_{a+b=k}(-1)^b\zeta(\{1\}^{p-1},a+2)\zeta(\{1\}^{q-1},b+2).
$$
\end{thmB}
We evaluate $\zeta^\star(r+2,\{2\}^m)$, for $r=0,1$, and $2$,
as applications of $G_{n+2}(p,q)$.
We show that the generating function of $\zeta^\star(r+2,\{2\}^m)$ is 
$$
\sum^\infty_{k=1}\frac1{k^{r+2}}\frac{\Gamma(k+x)\Gamma(k-x)}{\Gamma(k)^2}.
$$
From this generating function we have
$$
\zeta^\star(r+2,\{2\}^m)=\sum_{p+q=2m\atop{a+b=r}}
(-1)^{q+b}{p+b\choose p}G_{q+2}(p+b,a).
$$
Then we get a sum formula
which was first appeared in \cite[Theorem 2]{IKOO}
and then Zagier regained it in \cite{Zag}.
$$
\sum_{a+b=n}\left(2+\delta_{0a}\right)\zeta^\star(\{2\}^a,3,\{2\}^b)=
2(2n+2)\left(1-\frac1{2^{2n+2}}\right)\zeta(2n+3).
$$

Our paper is organized as follows. In Section 2, we present some basic properties of 
$G_{n+2}(p,q)$. We investigate the generating function of $\zeta^\star(r+2,\{2\}^m)$
in Section 3. When we calculate the values of $G_{n+2}(p,q)$ we need some
sum formulas of MZVs and MZSVs of height one. We write these sum formulas in 
Section 4. In order to evalute $\zeta^\star(r+2,\{2\}^n)$
we calculate some sum formulas of $G_{n+2}(p,q)$ in Section 5.
In the last section we evaluate the values of $\zeta^\star(r+2,\{2\}^n)$, for $r=0,1,2$,
as our applications.

\section{Properties of $G_{n+2}(p,q)$}
In this section we will give some interesting properties of $G_{n+2}(p,q)$.
We need a proposition in \cite{EieB} to transform a sum of MVZs to
a integral representation.
\begin{proposition}\cite[Proposition 6.5.1]{EieB}\label{pro2.1}
\begin{eqnarray*}
\lefteqn{\sum_{|\alpha|=m}\zeta(\{1\}^p,\alpha_1,\alpha_2,\ldots,\alpha_q+n)
=\frac1{p!(q-1)!(m-q)!(n-1)!}}\\
&\times&\int_{E_2}
\left(\log\frac1{1-t_1}\right)^p\left(\log\frac{t_2}{t_1}\right)^{m-q}
\left(\log\frac{1-t_1}{1-t_2}\right)^{q-1}\left(\log\frac1{t_2}\right)^{n-1}
\frac{dt_1dt_2}{(1-t_1)t_2}.
\end{eqnarray*}
\end{proposition}
The following we give two other representations of $G_{n+2}(p,q)$:
\begin{theorem}\label{thm2.2}
\begin{eqnarray}\label{eq2.2} 
G_{n+2}(p,q)&=&\sum^{p+q+1}_{r=p+1}{r-1\choose p}
\sum_{|\alpha|=p+q+1}\zeta(\alpha_1,\ldots,\alpha_r+n+1)\\
&=&\frac1{p!q!n!}\int_{E_2} \label{eq2.3} 
\left(\log\frac1{1-t_1}\right)^p\left(\log\frac1{1-t_2}\right)^q
\left(\log\frac{t_2}{t_1}\right)^n\frac{dt_1dt_2}{(1-t_1)t_2}.
\end{eqnarray}
\end{theorem}
\begin{proof}
Since $G_{n+2}(p,q)$ is a product of a multiple zeta value and 
a multiple zeta-star value, we use the shuffle product relation (ref. \cite{CCE618})
and then we get 
$$
G_{n+2}(p,q)=\sum^{p+q+1}_{r=p+1}{r-1\choose p}
\sum_{|\alpha|=p+q+1}\zeta(\alpha_1,\ldots,\alpha_r+n+1).
$$
Using Proposition \ref{pro2.1} we can transform this summation 
to an integral representation
\begin{eqnarray*}
\lefteqn{G_{n+2}(p,q)=
\sum^{p+q+1}_{r=p+1}{r-1\choose p}
\frac1{(r-1)!(p+q-r+1)!n!}}\\
&&\qquad\qquad\qquad\times\int_{E_2}
\left(\log\frac{t_2}{t_1}\right)^{p+q-r+1}
\left(\log\frac{1-t_1}{1-t_2}\right)^{r-1}
\left(\log\frac1{t_2}\right)^n\frac{dt_1dt_2}{(1-t_1)t_2}.
\end{eqnarray*}
Using the binomial theorem we have
$$
G_{n+2}(p,q)=
\frac1{p!q!n!}\int_{E_2}
\left(\log\frac{t_2}{t_1}+\log\frac{1-t_1}{1-t_2}\right)^q
\left(\log\frac{1-t_1}{1-t_2}\right)^p
\left(\log\frac1{t_2}\right)^n\frac{dt_1dt_2}{(1-t_1)t_2}.
$$
We change the variables (ref. \cite{CCE610})
$$
\frac{1-t_1}{1-t_2}=\frac1{1-u_1}\qquad\mbox{and}\qquad\frac1{t_2}=\frac{u_2}{u_1},
$$
then we get the final desired form.
\end{proof}
Therefore we can give a new integral represetation of $\zeta^\star(\{1\}^q,n+2)$.
\begin{corollary}\label{cor2.3} 
Let $q$ and $n$ be nonnegative integers. Then
\begin{eqnarray}\label{eq2.4} 
\zeta^\star(\{1\}^q,n+2) &=& \sum^{q+1}_{r=1}
\sum_{|\alpha|=q+1}\zeta(\alpha_1,\ldots,\alpha_r+n+1) \\
&=& \frac1{q!n!}\int_{E_2}\left(\log\frac1{1-t_2}\right)^q \label{eq2.5} 
\left(\log\frac{t_2}{t_1}\right)^n\frac{dt_1dt_2}{(1-t_1)t_2}.
\end{eqnarray}
\end{corollary}
The function $G_{n+2}(p,q)$ has a reflection formula as follows.
\begin{proposition}\label{pro2.4} 
For a pair of positive integers $p$, $q$, and an integer $k\geq 0$,
we have
\begin{equation}\label{eq2.6} 
G_{k+3}(p-1,q)+(-1)^kG_{k+3}(q-1,p)
=\sum_{a+b=k}(-1)^b\zeta(\{1\}^{p-1},a+2)\zeta(\{1\}^{q-1},b+2).
\end{equation}
\end{proposition}
\begin{proof}
Consider the double integral
$$
\frac1{k!p!q!}\int^1_0\int^1_0
\left(\log\frac ut\right)^k\left(\log\frac1{1-t}\right)^p
\left(\log\frac 1{1-u}\right)^q\frac{dt}t\frac{du}u.
$$
Replace the factor 
$$
\left(\log\frac ut\right)^k=\left(\log\frac1t-\log\frac1u\right)^k
$$
by its binomial expansion
$$
\sum_{a+b=k}\frac{k!(-1)^b}{a!b!}\left(\log\frac1t\right)^a
\left(\log\frac1u\right)^b
$$
we see immediate that its value is given by 
$$
\sum_{a+b=k}(-1)^b\zeta(\{1\}^{p-1},a+2)
\zeta(\{1\}^{q-1},b+2).
$$
Now we decompose the square $[0,1]\times [0,1]$ into 
union of two simplices 
$$
D_1:0<t<u<1 \quad\mbox{and}\quad D_2:0<u<t<1.
$$
On $D_1:0<t<u<1$, the corresponding integral 
$$
\frac1{p!q!k!}\int_{0<t<u<1}
\left(\log\frac1{1-t}\right)^p\left(\log\frac1{1-u}\right)^q
\left(\log\frac ut\right)^k\frac{dt}t\frac{du}u
$$
can be rewritten as 
$$
\frac1{(p-1)!q!(k+1)!}\int_{0<t<u<1}
\left(\log\frac1{1-t}\right)^{p-1}
\left(\log\frac1{1-u}\right)^q
\left(\log\frac ut\right)^{k+1}\frac{dt}{1-t}\frac{du}u
$$
which is equal to $G_{k+3}(p-1,q)$. In the same manner,
the corresponding integral on $D_2$ is $(-1)^kG_{k+3}(q-1,p)$.
\end{proof}
Furthermore, the values of $G_2(p,q)$ can be easily calculated.
\begin{proposition} \label{pro2.5} 
\begin{equation} \label{eq2.7} 
G_2(p,q)={p+q+1\choose q}\zeta(p+q+2)=G_2(q-1,p+1).
\end{equation}
\end{proposition}
\begin{proof}
Since $G_2(p,q)$ have the following integral representation
$$
\frac1{p!q!}\int_{E_2}\left(\log\frac1{1-t_1}\right)^p
\left(\log\frac1{1-t_2}\right)^q\frac{dt_1dt_2}{(1-t_1)t_2}.
$$
With $\log\frac1{1-t_2}=\log\frac{1-t_1}{1-t_2}+\log\frac1{1-t_1}$ 
we use the binomial theorem to decompose the second factor in the above representation,
then we have
$$
\sum^q_{j=0}{p+j\choose p}\zeta(\{1\}^{p+q},2).
$$
Since 
$$
\sum^q_{j=0}{p+j\choose p}={p+q+1\choose q}
$$
and using the dual theorem
$$
\zeta(\{1\}^{p+q},2)=\zeta(p+q+2)
$$
we conclude the first identity. Since ${p+q+1\choose q}={p+q+1\choose p+1}$
we get the second identity.
\end{proof}
There is an easy result:
$$
\zeta^\star(\{1\}^q,2)=G_2(0,q)=(q+1)\zeta(q+2).
$$
\section{The generating function of $\zeta^\star(r+2,\{2\}^m)$}
Let $r$ and $m$ be nonnegative integers. 
We define the generating function of $\zeta^\star(r+2,\{2\}^m)$ to be
\begin{equation}\label{eq3.1} 
G^\star_r(x):=\sum^\infty_{m=0}\zeta^\star(r+2,\{2\}^m)x^{2m}.
\end{equation}
This generating function can be expressed as the following.
\begin{proposition} \label{pro3.1} 
\begin{equation} \label{eq3.2} 
G^\star_r(x)=\sum^\infty_{k=1}\frac1{k^{r+2}}
\frac{\Gamma(k+x)\Gamma(k-x)}{\Gamma(k)^2}.
\end{equation}
\end{proposition}
\begin{proof}
From the well-known formula 
$$
\frac1{\Gamma(s+1)}=e^{\gamma s}\prod^\infty_{n=1}\left(1+\frac sn\right)e^{-\frac sn},
$$
we have
\begin{eqnarray}\nonumber
\frac{\Gamma(k+x)\Gamma(k-x)}{\Gamma(k)^2}
&=& \frac{k^2}{k^2-x^2}\frac{\Gamma(k+x+1)\Gamma(k-x+1)}{\Gamma(k+1)^2} \\
&=& \frac{k^2}{k^2-x^2}\prod^\infty_{n=1} \nonumber
\frac{\left(1+\frac kn\right)^2}{\left(1+\frac{k+x}n\right)\left(1+\frac{k-x}n\right)} \\
&=& \frac1{1-\left(\frac xk\right)^2}\prod^\infty_{n=1}\frac1{1-\frac{x^2}{(n+k)^2}} 
\ \ =\ \  \prod_{n\geq k}\left(1-\frac{x^2}{n^2}\right)^{-1}.\label{eq3.3}
\end{eqnarray}
Since
$$
\sum^\infty_{m=0}\zeta^\star(r+2,\{2\}^m)x^{2m}
=\sum^\infty_{k=1}\frac1{k^{r+2}}\prod_{n\geq k}
\left(1-\frac{x^2}{n^2}\right)^{-1}
=\sum^\infty_{k=1}\frac1{k^{r+2}}\frac{\Gamma(k+x)\Gamma(k-x)}{\Gamma(k)^2}.
$$
Thus we have this conclusion.
\end{proof}
This identitiy can give us the integral representation of $\zeta^\star(r+2,\{2\}^n)$.
\begin{theorem} \label{thm3.2} 
For integers $n, r\geq 0$, we have
\begin{eqnarray}\nonumber
\lefteqn{\frac1{r!n!}\int_{E_2}\left(\log\frac{1-t_1}{1-t_2}\right)^r
\left(\log\frac1{1-t_1}-\log\frac{t_2}{t_1}\right)^n\,\frac{dt_1dt_2}{(1-t_1)t_2}}\\
&&\qquad\qquad\qquad\qquad\qquad=\left\{\begin{array}{ll}
\zeta^\star(r+2,\{2\}^m),&\mbox{if }n=2m,\\
0,&\mbox{if }n=2m+1,\end{array}\right.\label{eq3.4}
\end{eqnarray}
where $E_2=\{(t_1,t_2)\in\mathbb R^2\mid 0<t_1<t_2<1\}$.
\end{theorem}
\begin{proof}
The double integrals come from the differentiation of the following function $H_r(x)$ wtih 
parameter $x>-1$, defined by 
$$
H_r(x)=\frac1{r!}\int_{E_2}\left(\log\frac{1-t_1}{1-t_2}\right)^r
\left(\frac{t_1}{t_2}\right)^x(1-t_1)^{-x}\frac{dt_1dt_2}{(1-t_1)t_2}.
$$
Indeed, we have 
$$
\frac1{m!}\left(\frac{d}{dx}\right)^{m}H_r(x)\Big|_{x=0}
=\frac1{r!m!}\int_{E_2}\left(\log\frac{1-t_1}{1-t_2}\right)^r
\left(\log\frac1{1-t_1}-\log\frac{t_2}{t_1}\right)^{m}\frac{dt_1dt_2}{(1-t_1)t_2},
$$
since 
$$
\frac{d}{dx}\left(\frac{t_1}{t_2}\right)^x(1-t_1)^{-x}
=\left(\log\frac1{1-t_1}-\log\frac{t_2}{t_1}\right)
\left(\frac{t_1}{t_2}\right)^x(1-t_1)^{-x}.
$$
Under the change of variables $u_1=1-t_2$ and $u_2=1-t_1$, $H_r(x)$ is transformed into 
$$
\frac1{r!}\int_{E_2}\left(\log\frac{u_2}{u_1}\right)^r\left(\frac{1-u_2}{1-u_1}\right)^x
u_2^{-x}\frac{du_1du_2}{(1-u_1)u_2}
$$
and it can be evaluated as 
$$
\sum^\infty_{k=1}\frac1{k^{r+2}}\frac{\Gamma(k+x)\Gamma(k-x)}{\Gamma(k)^2}.
$$
Along with our previous results, we conclude that 
\begin{equation}\label{eq3.5} 
G^\star_r(x)=H_r(x)=\frac1{r!}\int_{E_2}\left(\log\frac{1-t_1}{1-t_2}\right)^r
\left(\frac{t_1}{t_2}\right)^x(1-t_1)^{-x}\frac{dt_1dt_2}{(1-t_1)t_2}.
\end{equation}
Take the coefficients of $x^{2n}$ of both sides, we obtain our double integral
representation of $\zeta^\star(r+2,\{2\}^n)$. Also note that $H(x)$ is an even function of $x$,
so that its coefficients of odd powers vanish.
\end{proof}
We can express $\zeta^\star(r+2,\{2\}^m)$ as a sum of $G_{n+2}(p,q)$.
\begin{proposition} \label{pro3.3} 
\begin{equation}\label{eq3.6} 
\zeta^\star(r+2,\{2\}^m)
=\sum_{p+q=2m\atop{a+b=r}}
(-1)^{q+b}{p+b\choose p}G_{q+2}(p+b,a).
\end{equation}
\end{proposition}
\begin{proof}
The integral representation of $\zeta^\star(r+2,\{2\}^m)$
in Theorem \ref{thm3.2} is
$$
\zeta^\star(r+2,\{2\}^m)
=\frac1{r!(2m)!}\int_{E_2}\left(\log\frac{1-t_1}{1-t_2}\right)^r
\left(\log\frac1{1-t_1}-\log\frac{t_2}{t_1}\right)^{2m}
\frac{dt_1dt_2}{(1-t_1)t_2}.
$$
We use the binomial theorem to decompose the first factor 
in the above integral with 
$$
\log\frac{1-t_1}{1-t_2}=\log\frac1{1-t_2}-\log\frac1{1-t_1}.
$$
Then from the integral representation of $G_{n+2}(p,q)$
in Theorem \label{thm2.2}, we get the conclusion.
\end{proof}
\section{Sums of multiple zeta values of height one}
\begin{proposition}\label{pro4.1} 
For integers $n\geq 0$ and $s, r\geq 2$, we have
\begin{equation}\label{eq4.1} 
\sum_{a+b=n}\zeta^\star(\{s\}^a)\zeta(sb+r)
=\sum_{a+b=n}\zeta^\star(\{s\}^a,r,\{s\}^b).
\end{equation}
\end{proposition}
\begin{proof}
Let $F(x)$ be the generating function of $\zeta^\star(\{s\}^a,r,\{s\}^b)$, 
where $s,r\geq 2$, that is 
$$
F(x)=\sum_{a,b\geq 0}\zeta^\star(\{s\}^a,r,\{s\}^b)x^{(a+b)s}.
$$
From this definition we have
\begin{eqnarray*}
F(x) &=& \sum^\infty_{n=1}\prod_{0<k\leq n}\left(1-\frac{x^s}{k^s}\right)^{-1}\cdot
\frac{1}{n^r}\cdot\prod_{\ell\geq n}\left(1-\frac{x^s}{\ell^s}\right)^{-1} \\
&=& \prod^\infty_{k=1}\left(1-\frac{x^s}{k^s}\right)^{-1}\cdot
\sum^\infty_{n=1}\frac1{n^r}
\left(1-\frac{x^s}{n^s}\right)^{-1} \\
&=& \sum^\infty_{n=0}\zeta^\star(\{s\}^n)x^{ns}\cdot
\sum^\infty_{k=0}\zeta(sk+r)x^{ks}\\
&=&\sum_{a,b\geq 0}\zeta^\star(\{s\}^a)\zeta(sb+r)x^{(a+b)s}.
\end{eqnarray*}
\end{proof}
We write $\zeta(\{1\}^r,n+2-r)$ in its double integral form, then 
$$
\sum^n_{r=0}(-1)^{r+n}\zeta(\{1\}^r,n+2-r)
=\frac1{n!}\int_{E_2}\left(\log\frac1{1-t_1}-\log\frac{t_2}{t_1}\right)^n
\frac{dt_1dt_2}{(1-t_1)t_2}.
$$
Using Theorem \ref{thm3.2} we have the following proposition.
\begin{proposition} \label{pro4.2} 
Let $n$ be a nonnegative integer. Then
\begin{equation}\label{eq4.2} 
\sum^n_{r=0}(-1)^{r+n}\zeta(\{1\}^r,n+2-r)
=\left\{\begin{array}{ll}
\zeta^\star(\{2\}^{m+1}), &\mbox{if }n=2m,\\
0, &\mbox{if }n=2m+1.
\end{array}\right.
\end{equation}
\end{proposition}
\begin{proposition} \label{pro4.3} 
Let $n$ be a nonnegative integer. Then
\begin{equation}\label{eq4.3} 
\sum^n_{r=0}(-1)^{r+n}(r+1)\zeta(\{1\}^{r+1},n+2-r)
=\left\{\begin{array}{ll}
(m+1)\zeta^\star(\{2\}^{m+2}), &\mbox{if }n=2m+1,\\
\displaystyle\sum_{a+b=m}\zeta^\star(\{2\}^a,3,\{2\}^b), &\mbox{if }n=2m.
\end{array}\right.
\end{equation}
\end{proposition}
\begin{proof}
The alternating sum has the integral representation
$$
\frac1{n!}\int_{E_2}\left(\log\frac1{1-t_1}\right)
\left(\log\frac1{1-t_1}-\log\frac{t_2}{t_1}\right)^n
\frac{dt_1dt_2}{(1-t_1)t_2}.
$$
Consider the following function
$$
M(x,y)=\int_{E_2}\left(\frac{t_1}{t_2}\right)^x(1-t_1)^{-x-y}
\frac{dt_1dt_2}{(1-t_1)t_2}.
$$
It can be seen that 
$$
\frac1{n!}\left(\frac\partial{\partial x}\right)^n
\left(\frac\partial{\partial y}\right)M(x,y)\Big|_{x=0,y=0}
=\frac1{n!}\int_{E_2}\left(\log\frac1{1-t_1}\right)
\left(\log\frac1{1-t_1}-\log\frac{t_2}{t_1}\right)^n
\frac{dt_1dt_2}{(1-t_1)t_2}.
$$
With the change of variables $t_1=1-u_2$ and $t_2=1-u_1$, 
we have the dual integral representation
$$
\int_{E_2}\left(\frac{1-u_2}{1-u_1}\right)^xu_2^{-x-y}
\frac{du_1du_2}{(1-u_1)u_2}.
$$
It can be evaluated as
$$
M(x,y)=\sum^\infty_{k=1}
\frac{\Gamma(k+x)\Gamma(k-x-y)}{\Gamma(k+1)\Gamma(k+1-y)}.
$$
So that 
$$
\frac\partial{\partial y}M(x,y)\Big|_{y=0}
=\sum^\infty_{k=1}\frac{\Gamma(k+x)\Gamma(k-x)}{\Gamma(k+1)^2}
\cdot\left[\psi(k+1)-\psi(k-x)\right],
$$
where $\psi(x)$ is the digamma function 
$$
\psi(x)=\frac{d}{dx}\log\Gamma(x)=\frac{\Gamma'(x)}{\Gamma(x)}.
$$
By Eq.\,(\ref{eq3.3}) and for a positive integer $p$ 
$$
\left(\frac{d}{dx}\right)^p\psi(x)=(-1)^{p-1}p!\zeta(p+1;x),
$$
 we have
 \begin{eqnarray*}
\lefteqn{ \frac1{n!}\left(\frac\partial{\partial x}\right)^n
\left(\frac\partial{\partial y}\right)M(x,y)\Big|_{x=0,y=0}}\\
&=&
\sum^\infty_{k=1}\frac1{k^2}\sum^n_{r=0}{n\choose r}
\frac{d^r}{dx^r}\left(\prod_{n\geq k}\left(1-\frac{x^2}{n^2}\right)^{-1}\right)
\cdot
\frac{d^{n-r}}{dx^{n-r}}\left[\psi(k+1)-\psi(k-x)\right]\Big|_{x=0}\\
&=&\left\{\begin{array}{ll}
\displaystyle\sum^{m+1}_{r=0}\zeta^\star(\{2\}^r)\zeta(2m+4-2r),&\mbox{if }n=2m+1,\\
\displaystyle\sum^m_{r=0}\zeta^\star(\{2\}^r)\zeta(2m+3-2r),&\mbox{if }n=2m.
\end{array}\right.
\end{eqnarray*}
Applying Proposition \ref{pro4.1} we conclude the results.
\end{proof}
We use Eq.\,(\ref{eq3.4}) with $n=2m+1$ and decompose 
$\displaystyle\log\left(\frac{1-t_1}{1-t_2}\right)
=\log\frac1{1-t_2}-\log\frac1{1-t_1}$, we have
\begin{eqnarray*}
\lefteqn{\frac1{(2n+1)!}\int_{E_2}\left(\log\frac1{1-t_1}\right)
\left(\log\frac1{1-t_1}-\log\frac{t_2}{t_1}\right)^{2n+1}
\frac{dt_1dt_2}{(1-t_1)t_2}}\\
&=&\frac1{(2n+1)!}\int_{E_2}\left(\log\frac1{1-t_2}\right)
\left(\log\frac1{1-t_1}-\log\frac{t_2}{t_1}\right)^{2n+1}
\frac{dt_1dt_2}{(1-t_1)t_2}.
\end{eqnarray*}
In terms of multiple zeta values it is 
$$
\sum_{p+q=2n+1}(-1)^q(p+1)\zeta(\{1\}^{p+1},q+2)
=\sum_{p+q=2n+1}(-1)^qG_{q+2}(p,1).
$$
The alternating sum in the left is equal to 
$(n+1)\zeta^\star(\{2\}^{n+2})$ while the sum in the 
right is equal to 
\begin{eqnarray*}
\lefteqn{\sum_{p+q=2n+1}\zeta^\star(\{1\}^{p+1},q+2)
-\sum_{0\leq a+b\leq 2n}(-1)^b\zeta(a+2)\zeta(\{1\}^{2n-a-b},b+2)}\\
&=&\sum_{p+q=2n+2}\zeta^\star(\{1\}^p,q+2)-\sum_{c+d=n+1}\zeta^\star(\{2\}^c)\zeta(2d+2)\\
&=&\sum_{p+q=2n+2}\zeta^\star(\{1\}^p,q+2)-(n+2)\zeta^\star(\{2\}^{n+2}).
\end{eqnarray*}
This leads to 
\begin{equation}\label{eq4.4} 
\sum_{p+q=2n+2}\zeta^\star(\{1\}^p,q+2)=(2n+3)\zeta^\star(\{2\}^{n+2}).
\end{equation}
Since
$$
\zeta^\star(\{2\}^{n+2})=2\left(1-\frac1{2^{2n+3}}\right)\zeta(2n+4),
$$
we have
$$
\sum_{p+q=2n+2}\zeta^\star(\{1\}^p,q+2)
=2(2n+3)\left(1-\frac1{2^{2n+3}}\right)\zeta(2n+4).
$$
Hence we can conclude the above result as a proposition.
\begin{proposition} \label{pro4.4} 
For a nonnegative integer $n$, 
\begin{equation}\label{eq4.5} 
\sum_{p+q=2n+2}\zeta^\star(\{1\}^p,q+2)
=2(2n+3)\left(1-\frac1{2^{2n+3}}\right)\zeta(2n+4).
\end{equation}
\end{proposition}
As a matter of fact, 
this result is just a special case of \cite[Theorem 1]{AO} which was proved by Aoki and Oho.
\begin{equation}\label{eq4.6} 
\sum_{{\bf k}\in I_0(k,s)}\zeta^\star({\bf k})
=2{k-1\choose 2s-1}(1-2^{1-k})\zeta(k),
\end{equation}
where $I_0(k,s)$ is the set of admissible multi-indices
${\bf k}=(k_1,k_2,\ldots,k_n)$ with weight $k=k_1+k_2+\cdots+k_n$
and height $s=\#\{i\mid k_i>1\}$. We include this result as an application of the function 
$G_{n+2}(p,q)$.
\section{Sum formulas of $G_{n+2}(p,q)$}
\begin{proposition} \label{pro5.1} 
Let $n$ be a nonnegative integer. Then 
$$
\sum_{p+q=2n}(-1)^qG_{q+2}(p,2)
=\sum_{p+q=2n+1}G_{q+2}(1,p)
-\sum_{a+b=n}\zeta(1,2a+3)\zeta^\star(\{2\}^b).
$$
\end{proposition}
\begin{proof}
By Proposition \ref{pro2.4} and \ref{pro2.5}, the left hand side of the above identity is
\begin{eqnarray*}
\lefteqn{\sum_{p+q=2n}(-1)^qG_{q+2}(p,2)}\\
&=&\sum^{2n}_{q=0}G_{2+q}(1,2n+1-q)
-\sum^{2n}_{q=1}\sum^{q-1}_{a=0}
(-1)^{q-1-a}\zeta(1,a+2)\zeta(\{1\}^{2n-q},q+1-a).
\end{eqnarray*}
After a suitable change the variable of indices and their orders,
we can reach the following  
$$
\sum^{2n}_{q=0}G_{2+q}(1,2n+1-q)
-\sum^{2n-1}_{\ell=0}\sum^{\ell}_{r=0}
(-1)^{r+\ell}\zeta(1,2n+1-\ell)\zeta(\{1\}^{r},\ell+2-r).
$$
By Proposition \ref{pro4.2} we have
$$
\sum_{p+q=2n}(-1)^qG_{q+2}(p,2)
=\sum^{2n}_{q=0}G_{2+q}(1,2n+1-q)
-\sum^n_{a=1}\zeta^\star(\{2\}^a)\zeta(1,2n+3-2a).
$$
Since $G_{2n+3}(1,0)=\zeta(1,2n+3)$, hence we complete the proof.
\end{proof}
\begin{proposition} \label{pro5.2} 
\begin{eqnarray*}
\lefteqn{\sum_{p+q=2n}(-1)^q(p+1)G_{q+2}(p+1,1)
=\sum_{p+q=2n}(p+1)\zeta^\star(\{1\}^{p+2},q+2)} \\
&&-\sum_{a+b=n}b\cdot\zeta(2a+2)\zeta^\star(\{2\}^{b+1})
-\sum_{a+b+c=n-1}\zeta^\star(\{2\}^a,3,\{2\}^b)\zeta(2c+3).
\end{eqnarray*}
\end{proposition}
\begin{proof}
We first use Proposition \ref{pro2.4} and \ref{pro2.5} 
such that the sum of $G_{q+2}(p+1,1)$ in the left hand side of the 
above identity becomes
$$
\sum_{p+q=2n}(p+1)G_{q+2}(0,p+2)
-\sum^{2n}_{q=1}\sum^{q-1}_{a=0}(-1)^{q-1-a}(2n+1-q)\zeta(a+2)
\zeta(\{1\}^{2n+1-q},q+1-a).
$$
Since $G_{q+2}(0,p+2)=\zeta^\star(\{1\}^{p+2},q+2)$,
we get the desired first factor in the right hand side of the above identity. 

After a suitable change the variable of indices and their orders
in the second factor, this factor can reach the following  
$$
\sum^{2n-1}_{\ell=0}\zeta(2n+1-\ell)
\sum^\ell_{r=0}(-1)^{\ell+r}(r+1)\zeta(\{1\}^{r+1},\ell+2-r).
$$
Using Proposition \ref{pro4.3} we get the desired remaining terms.
\end{proof}
\begin{theorem} \label{thm5.3} 
For a nonnegative integer $n$, we have
\begin{eqnarray*}
\lefteqn{\sum_{p+q=2n+1}(p+1)\zeta^\star(\{1\}^{p+1},q+2)
-\sum_{p+q=2n+1}G_{q+2}(1,p)}\\
&=&{2n+4\choose 3}\zeta(2n+4)+\sum^n_{j=1}(-1)^j
\sum_{|c_j|=2n+1-2j}\zeta(c_{j0}+3,c_{j1}+2,\ldots,c_{jj}+2)W(c_j)
\end{eqnarray*}
with 
$$
W(c_j)=W(c_{j0},c_{j1},\ldots,c_{jj})={c_{j0}+3\choose 3}(c_{j1}+1)\cdots(c_{jj}+1).
$$
\end{theorem}
\begin{proof}
The difference in the left hand side has the integral representation
$$
\frac1{(2n+1)!}\int_{E_2}\left(\log\frac{1-t_1}{1-t_2}\right)
\left(\log\frac1{1-t_2}+\log\frac{t_2}{t_1}\right)^{2n+1}\frac{dt_1dt_2}{(1-t_1)t_2}.
$$
So we begin with its generating function
$$
G(x,y)=\int_{E_2}\left(\frac{t_1}{t_2}\right)^x(1-t_1)^{-y}(1-t_2)^{x+y}
\frac{t_1dt_2}{(1-t_1)t_2}.
$$
Indeed, our integral is equal to 
$$
\frac1{(2n+1)!}\left(\frac{\partial}{\partial x}\right)^{2n+1}\frac\partial{\partial y}
G(x,y)\Big|_{x=y=0}.
$$
We change the variables $t_1=1-u_2$, $t_2=1-u_1$,
the dual form of $G(x,y)$ is given by 
$$
\int_{E_2}\left(\frac{1-u_2}{1-u_1}\right)^xu_1^{x+y}u_2^{-y}
\frac{du_1du_2}{(1-u_1)u_2}
$$
and it can be evaluated as 
$$
\sum_{k=1}^\infty\frac1{(k+x+y)}\frac{\Gamma(k+x)^2}{\Gamma(k)\Gamma(k+2x+1)}.
$$
As
$$
\left(-\frac\partial{\partial y}\right)G(x,y)\Big|_{y=0}
=\sum^\infty_{k=1}\frac1{(k+x)^2}\frac{\Gamma(k+x)^2}{\Gamma(k)\Gamma(k+2x+1)}
$$
so our evaluation is equivalent to find the coefficient of $x^{2n+1}$ of the above
function up to the sign $(-1)^{2n+1}$. To do so we have to express the quotient of 
gamma functions as an infinite product through the infinite product formula of the 
gamma function. The procedure is
\begin{eqnarray*}
\sum^\infty_{k=1}\frac1{(k+x)^2(k+2x)}\frac{\Gamma(k+x)^2}{\Gamma(k)\Gamma(k+2x)}
&=& \sum^\infty_{k=1}\frac{k}{(k+x)^4}\frac{\Gamma(k+x+1)^2}{\Gamma(k+1)\Gamma(k+2x+1)}\\
&=& \sum^\infty_{k=1}\frac{k}{(k+x)^4}
\prod^\infty_{v=1}\frac{(1+k/v)(1+(k+2x)/v)}{(1+(k+x)/v)^2} \\
&=& \sum^\infty_{k=1}\frac{k}{(k+x)^4}
\prod^\infty_{v=1}\frac{(1+2x/(v+k))}{(1+x/(v+k))^2}\\
&=& \sum^\infty_{k=1}\frac{k}{(k+x)^4}
\prod^\infty_{v=1}\left\{1-\frac{x^2}{(v+k)^2}\left(1+\frac{x}{v+k}\right)^{-2}\right\}.
\end{eqnarray*}
From the final infinite product its coefficient of $(-x)^{2n+1}$ is given by
$$
{2n+4\choose 3}\zeta(2n+4)+\sum^n_{j=1}(-1)^j\sum_{|c_j|=2n+1-2j}
\zeta(c_{j0}+3,c_{j1}+2,\ldots,c_{jj}+2)W(c_j).
$$
\end{proof}
\section{Evaluations of $\zeta^\star(r+2,\{2\}^n)$}
In this section we give the evaluations of $\zeta^\star(r+2,\{2\}^n)$,
for $r=0$, $1$, and $2$ using the formula in Proposition \ref{pro3.3}.

Firstly, let $r=0$. Then we have
\begin{eqnarray*}
\zeta^\star(2,\{2\}^n)&=&\sum_{p+q=2n}(-1)^qG_{q+2}(p,0) \\
&=&\sum_{p+q=2n}(-1)^q\zeta(\{1\}^p,q+2).
\end{eqnarray*}
This is exactly the even case in Proposition \ref{pro4.2}.

Secondly, let $r=1$. Then we have
$$
\zeta^\star(3,\{2\}^n) =H_0+H_1,
$$
where 
\begin{eqnarray*}
H_0 &=& \sum_{p+q=2n}(-1)^{q+1}(p+1)G_{q+2}(p+1,0),
\qquad\mbox{and}\\
H_1&=& \sum_{p+q=2n}(-1)^qG_{q+2}(p,1).
\end{eqnarray*}
Since $G_{q+2}(p,0)=\zeta(\{1\}^p,q+2)$, $H_0$ is the negative of 
the even case in Proposition \ref{pro4.3}.
$$
H_0=-\sum_{a+b=n}\zeta^\star(\{2\}^a,3,\{2\}^b).
$$
On the other hand, we use a similar method as the proof of Proposition 
\ref{pro5.1} to treat $H_1$.
\begin{eqnarray*}
H_1&=&\sum^{2n}_{q=0}G_{q+2}(0,2n+1-q)
-\sum^{2n}_{q=1}\sum^{q-1}_{a=0}
(-1)^{q-1-a}\zeta(a+2)\zeta(\{1\}^{2n-q},q+1-a).
\end{eqnarray*}
After a suitable change the variable of indices and their orders,
we can reach the following  
$$
\sum_{p+q=2n+1\atop{p\geq 1}}\zeta^\star(\{1\}^p,q+2)
-\sum^{2n-1}_{\ell=0}\sum^{\ell}_{r=0}
(-1)^{r+\ell}\zeta(2n+1-\ell)\zeta(\{1\}^{r},\ell+2-r).
$$
By Proposition \ref{pro4.2} we have
$$
H_1
=\sum_{p+q=2n+1}\zeta^\star(\{1\}^p,q+2)
-\sum^n_{a=0}\zeta^\star(\{2\}^a)\zeta(2n+3-2a).
$$
We use Eq.\,(\ref{eq4.6}) to write the first term as 
a Riemann zeta value times a constant and 
apply Proposition \ref{pro4.1} to the second term, then we have
$$
H_1=
2(2n+2)\left(1-\frac1{2^{2n+2}}\right)\zeta(2n+3)
-\sum_{a+b=n}\zeta^\star(\{2\}^a,3,\{2\}^b).
$$
Hence we conclude the value of $\zeta^\star(3,\{2\}^n)$ as
$$
\zeta^\star(3,\{2\}^n)=
2(2n+2)\left(1-\frac1{2^{2n+2}}\right)\zeta(2n+3)
-2\sum_{a+b=n}\zeta^\star(\{2\}^a,3,\{2\}^b).
$$
Note that the form $\zeta^\star(3,\{2\}^n)$ is the special 
case in the last sum with $a=0$, $b=n$. Thus we can collect 
them together and give a beautiful sum formula which was first appeared in \cite[Theorem 2]{IKOO}
and then Zagier regained it in \cite{Zag}.
\begin{equation}\label{eq6.1} 
\sum_{a+b=n}\left(2+\delta_{0a}\right)\zeta^\star(\{2\}^a,3,\{2\}^b)=
2(2n+2)\left(1-\frac1{2^{2n+2}}\right)\zeta(2n+3).
\end{equation}

Thirdly, let $r=2$. Then we have
$$
\zeta^\star(4,\{2\}^n)=A_0-A_1+A_2,
$$
where 
\begin{eqnarray*}
A_0 &=& \sum_{p+q=2n}(-1)^q{p+2\choose 2}\zeta(\{1\}^{p+2},q+2),\\
A_1 &=& \sum_{p+q=2n}(-1)^q(p+1)G_{q+2}(p+1,1),\qquad\mbox{and}\\
A_2 &=& \sum_{p+q=2n}(-1)^qG_{q+2}(p,2).
\end{eqnarray*}
Proposition \ref{pro5.2} gives an evaluation of $A_1$ and makes it to be
\begin{eqnarray*}
\lefteqn{A_1=\sum_{p+q=2n}(p+1)\zeta^\star(\{1\}^{p+2},q+2)} \\
&&-\sum_{a+b=n}b\cdot\zeta(2a+2)\zeta^\star(\{2\}^{b+1})
-\sum_{a+b+c=n-1}\zeta^\star(\{2\}^a,3,\{2\}^b)\zeta(2c+3).
\end{eqnarray*}
We decompose the first term as
\begin{eqnarray*}
\lefteqn{\sum_{p+q=2n}(p+1)\zeta^\star(\{1\}^{p+2},q+2)}\\
&=&\sum_{p+q=2n+1}(p+1)\zeta^\star(\{1\}^{p+1},q+2)
-\sum_{p+q=2n+2}\zeta^\star(\{1\}^p,q+2)+\zeta(2n+4).
\end{eqnarray*}
Therefore $A_1$ beomes
\begin{eqnarray*}
\lefteqn{A_1=\sum_{p+q=2n+1}(p+1)\zeta^\star(\{1\}^{p+1},q+2)
-\sum_{p+q=2n+2}\zeta^\star(\{1\}^p,q+2)+\zeta(2n+4)}\\
&&-\sum_{a+b=n}b\cdot\zeta(2a+2)\zeta^\star(\{2\}^{b+1})
-\sum_{a+b+c=n-1}\zeta^\star(\{2\}^a,3,\{2\}^b)\zeta(2c+3).
\end{eqnarray*}
The value of $A_2$ is calculated in Proposition \ref{pro5.1}:
$$
A_2=\sum_{p+q=2n+1}G_{q+2}(1,p)
-\sum_{a+b=n}\zeta(1,2a+3)\zeta^\star(\{2\}^b).
$$
We put the first terms of $A_1$ and $A_2$ together, we have
$$
\sum_{p+q=2n+1}G_{q+2}(1,p)-
\sum_{p+q=2n+1}(p+1)\zeta^\star(\{1\}^{p+1},q+2).
$$
This value is just the left hand side of the identity in Theorem \ref{thm5.3}.
Applying the result of Theorem \ref{thm5.3} and Proposition \ref{pro4.4}, 
we get the final form of $\zeta^\star(4,\{2\}^n)$ as
\begin{eqnarray*}
\zeta^\star(4,\{2\}^n) &=&
-{2n+4\choose 3}\zeta(2n+4)+2(2n+3)\left(1-\frac1{2^{2n+3}}\right)\zeta(2n+4)-\zeta(2n+4)\\
&&+\sum^n_{j=1}(-1)^{j+1}\sum_{|c_j|=2n+1-2j}
\zeta(c_{j0}+3,c_{j1}+2,\ldots,c_{jj}+2)W(c_j)\\
&&+\sum_{p+q=2n}(-1)^q{p+2\choose 2}\zeta(\{1\}^{p+2},q+2)
-\sum_{a+b=n}\zeta(1,2a+3)\zeta^\star(\{2\}^b)\\
&&+\sum_{a+b=n}b \zeta(2a+2)\zeta^\star(\{2\}^{b+1})
+\sum_{a+b+c=n-1}\zeta^\star(\{2\}^a,3,\{2\}^b)\zeta(2c+3).
\end{eqnarray*}
%
%
\section*{Acknowledgements}
Chen was funded by the Ministry of Science and Technology, Taiwan, Republic of
China, through grant MOST 106-2115-M-845-001.
%
%


\begin{thebibliography}{99}

\bibitem{AO}
T.~Aoki, Y.~Ohno,
Sum relations for multiple zeta values and connection formulas for the Gauss 
hypergeometric functions,
\textsl{Publ. RIMS, Kyoto Univ.},
41 (2005), 329--337.

\bibitem{Chen}
K.-W.~Chen,
Generalized harmonic numbers and Euler sums,
\textsl{Int. J. Number Theory},
13 (2) (2017), 513--528. DOI:10.1142/S1793042116500883.

\bibitem{CCE610}
K.-W.~Chen, C.-L.~Chung, M.~Eie,
Sum formulas and duality theorems of multiple zeta values,
\textsl{J. Number Theory},
158 (2016), 33--53.

\bibitem{CCE618}
K.-W.~Chen, C.-L.~Chung, M.~Eie,
Combinatorial implications of decomposition theorems of multiple zeta values,
\textsl{J. Comb. Number Theory},
7 (2) (2016), 111--129.

\bibitem{Choi}
J.~Choi,
Summation formulas involving binomial coefficients, harmonic numbers, 
and generalized harmonic numbers,
\textsl{Abst. Appl. Anal.},
2014 (2014), Article ID 501906, 10 pages.

\bibitem{CC1}
M.-A.~Coppo, B.~Candelpergher,
The Arakawa-Kaneko zeta function,
\textsl{Ramanujan J.},
22 (2) (2010), 153--162. DOI: 10.1007/S11139-009-9205-X.

\bibitem{CC2}
M.-A.~Coppo, B.~Candelpergher,
Inverse binomial series and values of Arakawa-Kaneko zeta functions,
\textsl{J. Number Theory},
150 (2015), 98--119.

\bibitem{EieB}
M.~Eie,
\textsl{Topics in Number Theory},
Monographs in Number Theory, vol. 2, World Scientific, Singapore, 2009.

\bibitem{Hoff0}
M.~E.~Hoffman,
Multiple harmonic series,
\textsl{Pac. J. Math.},
152 (1992), 275--290.

\bibitem{Hoff}
M.~E.~Hoffman,
Harmonic-number summation identities,
symmetric functions, and multiple zeta values,
\text{Ramanujan J.},
42 (2017), 501--526.

\bibitem{IKOO}
K.~Ihara, J.~Kajikawa, Y.~Ohno, J.-I.~Okuda,
Multiple zeta values vs. multiple zeta-star values,
\textsl{J. Algebra},
332 (2011), 187--208.

\bibitem{Mun}
S.~Muneta,
Algebraic setup of non-strict multiple zeta values,
\textsl{Acta Arith.},
136 (2009), 7--18.

\bibitem{Zag0}
D.~Zagier,
Values of zeta functions and their applications,
\textsl{In: First European Congres of Mathematics,
vol. II, Paris, 1992, Progr. Math. 120. Birkhuser, Basel 1994},
(1992), 497--512.

\bibitem{Zag}
D.~Zagier, 
Evaluation of the multiple zeta value $\zeta(2,\ldots,2,3,2,\ldots,2)$,
\textsl{Annals of Mathematics},
175 (2012), 977-1000.

\bibitem{Zhao}
J.~Zhao,
\textsl{Multiple Zeta Functions, Multiple Polylogarithms and Their 
Special Values},
Series on Number Theory and Its Applications,
World Scientific Publishing Co. Pte. Ltd, 2016.
\end{thebibliography}
\end{document}